%% file: sign.tex
\documentclass[11pt]{article}

\usepackage[a4paper]{geometry}
\usepackage{indentfirst}
\usepackage{hyperref}
\usepackage[anythingbreaks]{breakurl}
\usepackage{amssymb, amsmath, amsthm}
\usepackage{graphicx}
\usepackage[labelsep=none]{caption}
\usepackage[labelformat=simple]{subcaption}

\hypersetup{bookmarksnumbered, pdfstartview=FitH, pdfpagemode=UseNone}
\allowdisplaybreaks
\newcommand{\mymod}[1]{\mkern 8mu (\mathrm{mod} \mkern 6mu #1)}

\newcommand{\mynewtheorem}[2]{\newtheorem{#1}{\indent #2}}
\mynewtheorem{question}{Question}

\mynewtheorem{lemma}{Lemma}
\mynewtheorem{proposition}{Proposition}
\mynewtheorem{theorem}{Theorem}
\mynewtheorem{corollary}{Corollary}
\newenvironment{myproof}[1][Proof]{\begin{proof}[\indent #1]}{\end{proof}}

\begin{document}

\title{\textbf{Slitherlink Signatures}}
\author{Nikolai Beluhov}
\date{}

\maketitle

\begin{abstract} Let $G$ be a planar graph and let $C$ be a cycle in $G$. Inside of each finite face of $G$, we write down the number of edges of that face which belong to $C$. This is the signature of $C$ in $G$. The notion of a signature arises naturally in the context of Slitherlink puzzles. The signature of a cycle does not always determine it uniquely. We focus on the ambiguity of signatures in the case when $G$ is a rectangular grid of unit square cells. We describe all grids which admit an ambiguous signature. For each such grid, we then determine the greatest possible difference between two cycles with the same signature on it. We also study the possible values of the total number of cycles which fit a given signature. We discuss various related questions as well. \end{abstract}

\input{sign-01-intro}

\input{sign-02-prelim}

\input{sign-03-trans}

\input{sign-04-even}

\input{sign-05-const}

\input{sign-06-size}

\input{sign-07-div}

\input{sign-08-diff}

\input{sign-09-mult}

\input{sign-10-further}

\input{sign-11-ack}

\input{sign-12-refs}
\end{document}

%% file: sign-01-intro.tex
\section{Introduction} \label{intro}

Let $G$ be a planar graph. We assume that $G$ is finite, connected, and without bridges. So every edge of $G$ is part of a cycle in $G$ and separates two distinct faces of $G$.

Let $C$ be a cycle in $G$. Inside of each finite face $F$ of $G$, we write down the number of edges of $F$ which belong to $C$. The resulting assignment of numbers to faces is the \emph{signature} of $C$ in $G$.

One might expect at first that the signature of a cycle determines it uniquely. However, this is not always true. We call a signature \emph{ambiguous} when it corresponds to two or more distinct cycles.

\begin{figure}[ht] \null \hfill \begin{subfigure}{35pt} \centering \includegraphics{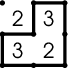} \caption{} \label{amb-a} \end{subfigure} \hspace{30pt} \begin{subfigure}{35pt} \centering \includegraphics{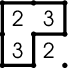} \caption{} \label{amb-b} \end{subfigure} \hfill \null \caption{} \label{amb-fig} \end{figure}

For example, Figure \ref{amb-fig} shows two distinct cycles in the $3 \times 3$ grid graph with the same signature.

The notion of a signature and the question of the ambiguity of signatures both arise naturally in the context of Slitherlink puzzles. A Slitherlink puzzle takes place on a rectangular grid of unit square cells. Some of the cells (though likely not all of them) contain a numerical clue. The solver's task is to draw a closed loop travelling along the grid lines so that, for each clued cell, the clue equals the number of edges of that cell which are part of the loop. In order to be considered sound, a Slitherlink puzzle must admit a unique solution.

Suppose that we are trying to design a sound Slitherlink puzzle whose solution is some predetermined loop. One way to go about it would be to simply clue all cells in accordance with that loop and then hope for the best. Should the resulting puzzle turn out way too easy, we might look for some superfluous clues to delete, so as to make it harder. On the other hand, we might sometimes discover that, even when all cells are clued, our loop is not the puzzle's only solution. This would mean that our loop can never be the solution to a sound Slitherlink puzzle.

Because of this connection, when $G$ is a grid graph, sometimes we refer to its signatures as \emph{Slitherlink} signatures. For the most part, this is the sort of signatures that we are going to study. However, occasionally we will touch upon the signatures of general planar graphs, too.

Let $S$ be the signature of $C$ in $G$. Following Slitherlink terminology, we refer to the number which $S$ assigns to a face $F$ of $G$ as the \emph{clue} of $F$ in $S$. We also say that $C$ is a \emph{solution} to $S$, or that $C$ \emph{satisfies} $S$. Though note that we measure grid sizes differently from the Slitherlink convention, by counting vertices instead of cells. So what would be called an $m \times n$ grid in standard Slitherlink terminology is, to us, an $(m + 1) \times (n + 1)$ grid instead.

The observation that a Slitherlink signature can be ambiguous is in \cite{B1}, with examples of sizes ranging from $3 \times 3$ to $8 \times 8$.

The series of exercises 410--420 in \cite{K1} explores the computational and mathematical aspects of Slitherlink puzzles. Exercise 411 asks if ambiguity is possible when every cell contains a clue, and its solution gives examples of sizes $3 \times 3$, $6 \times 6$, and $7 \times 7$. The term ``signature'' is used in this sense in the solution to exercise 415, where all Slitherlink signatures on the $6 \times 6$ grid are generated as a means towards the end of analysing the sound Slitherlink puzzles of that size which satisfy certain constraints. The idea of generalising the notion of a signature from Slitherlink puzzles to arbitrary planar graphs was suggested to the author by Knuth, in correspondence regarding that series of exercises.

Further connections between \cite{B1} and \cite{K1} and the present work are listed below.

Given a planar graph $G$, some of the most natural questions that we can ask about its signatures are as follows:

\begin{question} \label{qa} Does the signature of a cycle in $G$ always determine it uniquely? Or, equivalently, do there exist two cycles in $G$ with the same signature? \end{question}

Since a signature contains quite a lot of information, we might expect two solutions to the same signature to agree over most of $G$, in the sense that for most edges $e$ of $G$ either $e$ belongs to both solutions, or to neither. Then we might also wonder about the greatest extent to which two such solutions can differ. To formalise this notion, we define the \emph{difference} of two cycles in $G$ to be the set of all edges $e$ of $G$ where they disagree, so that $e$ is part of one of the two cycles but not the other.

\begin{question} \label{qb} What is the greatest possible size of the difference between two cycles with the same signature in $G$? \end{question}

We define the \emph{multiplicity} of a signature $S$ to be the total number of solutions to $S$.

\begin{question} \label{qc} What is the greatest possible multiplicity of a signature in $G$? Or, more generally: What are all possible values of the multiplicity of a signature in $G$? \end{question}

We can ask each one of these questions also for a class of graphs $\mathcal{G}$ rather than an individual graph $G$. For example, the second part of Question \textbf{\ref{qc}} would then become ``What are all possible values of the multiplicity of a signature in a graph of $\mathcal{G}$?''.

We go on to an overview of the contents of the rest of the paper.

Section \ref{prelim} lists some terms and notations.

Section \ref{trans} introduces one family of auxiliary graphs which we then use to answer the following question:

The definition of a signature singles out the exterior face -- it is the only one without a clue. This makes a lot of sense with Slitherlink puzzles since on grids the finite faces are all alike and the exterior face is very different from them. In the context of general planar graphs, though, the faces are all on equal footing -- for each face of $F$ of $G$, we can redraw $G$ isomorphically so that $F$ becomes the exterior one. Do we gain any additional information by clueing the exterior face as well?

This question was posed by Knuth in correspondence with the author. The author found a proof that no, we do not, and published the result as a journal problem. \cite{B2} The proof relies on the aforementioned family of auxiliary graphs, and has been reproduced in Section \ref{trans}. Equivalently, two cycles with the same signature are always of the same length. This is Theorem \ref{len}, and the original result is Corollary \ref{ext}.

Section \ref{even} introduces one family of edge subsets in planar graphs. The reason we consider them is that the difference of two cycles with the same signature is always of this kind. So all results that we establish about such subsets apply also to such differences. This family of edge subsets turns out to be isomorphic to one family of vertex subsets which has already been studied in the literature; see Section \ref{even} for the details.

Lemma \ref{4uv} and Corollary \ref{4uvmn} of Section \ref{even} are of particular interest. They show that the size of such a subset on a given rectangular grid can only take on some narrowly constrained values. Theorem \ref{kss} is the vertex subset analogue.

Section \ref{const} collects a number of constructions of ambiguous signatures. We rely on them for the existence parts of the main results of Sections \ref{size} and \ref{diff}.

Section \ref{size} completely resolves Question \textbf{\ref{qa}} on rectangular grids. The answer is given by Theorem \ref{ambs}.

The author reported this result and its proof in correspondence with Knuth. Subsequently, the sufficiency of condition (b) of Theorem \ref{ambs} for uniqueness was cited in the solution to exercise 411. Furthermore, Proposition \ref{22} was reproduced in \cite{K1} as exercise~420.

Condition (b) of Theorem \ref{ambs} strongly resembles one earlier uniqueness result, Theorem~2 in \cite{GL}, which is about a certain family of Minesweeper puzzles. Ultimately, both~that Minesweeper theorem and condition (b) of Theorem \ref{ambs} are corollaries of one fundamental fact about the isomorphic families of edge and vertex subsets discussed in Section \ref{even}. We spell out the details there.

Section \ref{div} establishes Theorem \ref{d8}. It is a divisibility result for the possible values of the size of the difference between two cycles with the same signature on rectangular grids.

Section \ref{diff} completely resolves Question \textbf{\ref{qb}} on rectangular grids. The answer is given by Theorems \ref{gds} and \ref{gdr}; the two handle the square case and the general rectangular case, respectively. The formula in the latter theorem yields the correct value of $0$ on grids where ambiguity is impossible, and so Theorem \ref{ambs} can be viewed as a special case of Theorem \ref{gdr}.

Section \ref{mult} explores Question \textbf{\ref{qc}}. In contrast to Questions \textbf{\ref{qa}} and \textbf{\ref{qb}}, we do not consider this question on each individual grid; instead, we consider it for the class of all rectangular grids as a whole. Still, even in this relaxed form it does not appear to be very easy, and our results on it are only partial.

By way of a warm-up, we resolve Question \textbf{\ref{qc}} completely for the class of all planar graphs. Theorem \ref{mpg} shows that, in this setting, every positive integer occurs as a multiplicity. We then look into one variant of Slitherlink with much weaker constraints which nevertheless provides some insights into the behaviour of the full ruleset. Here, too, every positive integer occurs as a multiplicity; this is Theorem \ref{mp}.

It is straightforward to construct infinitely many Slitherlink signatures with multiplicities $1$ and $2$. One example of a Slitherlink signature with multiplicity $4$, of size $7 \times 7$, is in \cite{B1} as well as in the solution to exercise 411 of \cite{K1}. Proposition \ref{m4} shows that in fact there exist infinitely many such signatures on square grids. The question of whether multiplicity $3$ is possible was posed by Bryce Herdt in a comment below \cite{B1} as well as by Knuth in the solution to exercise 411. Proposition \ref{m3} shows that, once again, there exist infinitely many such signatures on square grids.

Finally, Section \ref{further} collects some open problems and suggestions for further directions of research.

%% file: sign-02-prelim.tex
\section{Preliminaries} \label{prelim}

We write $A \Delta B$ for the symmetric difference of two sets $A$ and $B$. In terms of the sets' indicator vectors, this is simply addition over $\mathbb{F}_2$. So $A_1 \Delta A_2 \Delta \cdots \Delta A_k$ is well-defined for any number of sets; it is the set of all elements which are in an odd number of sets out of $A_1$, $A_2$, $\ldots$, $A_k$.

Let $G$ be a planar graph and let $G_1$ and $G_2$ be two subgraphs of $G$. We write $G_1 \Delta G_2$ for the symmetric difference of the edge sets of $G_1$ and $G_2$. (When both of $G_1$ and $G_2$ are cycles, this is the difference of $G_1$ and $G_2$ as defined in the introduction.) We also use $G_1 = G_2 \Delta E$ as a synonym for $E = G_1 \Delta G_2$.

A \emph{pseudocycle} in $G$ is the union of some number of pairwise disjoint cycles in $G$. So locally a pseudocycle looks a lot like a true cycle, in that it is a $2$-regular subgraph of~$G$. Globally, however, a pseudocycle need not be connected. A pseudocycle which fits a signature $S$ or, more generally, a Slitherlink puzzle $P$ is known as a \emph{weak solution} to $S$ or $P$. (Exercise 412 of \cite{K1}.)

We can define a \emph{pseudosignature} relative to a pseudocycle in the exact same way as a signature is defined relative to a true cycle. The main results of Sections \ref{trans}, \ref{size}, \ref{div}, and \ref{diff} continue to hold in the more general setting of pseudosignatures and their weak solutions. The proofs require only minor adjustments, or in some cases none at all.

We assume polygons in the plane to be closed, so that they include their boundaries. For example, when we say that an edge of some planar graph is contained inside of some polygon, we allow a vertex of that edge to lie on the boundary of the polygon.

To us, the grid of size $m \times n$ is a planar graph embedded in the plane so that its vertices are the integer points $(x, y)$ with $1 \le x \le m$ and $1 \le y \le n$ and its edges are the unit segments between these points. Throughout the rest of this section, let $G$ denote this particular planar graph.

We label each cell of $G$ with the coordinates of its center. So the vertices of the cell $(x, y)$ of $G$ are $(x \pm 1/2, y \pm 1/2)$, where we take all four combinations of signs. Notice that the coordinates of the vertices of $G$ are integers but the coordinates of the cells of $G$ are half-integers.

We call a vertex $(x, y)$ of $G$ \emph{even} or \emph{odd} depending on the parity of $x + y$. Similarly, we call a cell $(x, y)$ of $G$ \emph{even} or \emph{odd} depending on the parity of $\lfloor x \rfloor + \lfloor y \rfloor$. The even and odd vertices of $G$ form a checkerboard pattern, and so do the even and odd cells of $G$ as well.

When we describe certain sets of cells, it will be convenient to consider congruences involving half-integers. Given two half-integers $x$ and $y$ and an integer $d$, we write $x \equiv y \pmod d$ when $d$ divides the integer $x - y$.

We use each one of the notations $pq$ and $p$---$q$ both for the edge joining the vertices $p$ and $q$ of a graph and for the straight-line segment joining the points $p$ and $q$ in the plane. If necessary, we specify whether we mean an edge or a segment.

When a path $p_1p_2 \ldots p_k$ in $G$ proceeds along a straight line, so that $p_2$, $p_3$, $\ldots$, $p_{k - 1}$ lie in this order on the straight-line segment $p_1p_k$, for brevity we write just $p_1{\sim}p_k$, without listing all of the intermediate vertices explicitly.

The \emph{lower side} of $G$ is the subgraph $(1, 1){\sim}(m, 1)$ of $G$. The \emph{left}, \emph{right}, and \emph{upper sides} of $G$ are defined analogously.

For brevity, we say just ``the center of $G$'' instead of ``the center of symmetry of $G$''. This is the point $((m + 1)/2, (n + 1)/2)$. Notice that the center of $G$ need not be a vertex of $G$.

The \emph{frame} of $G$ is the axis-aligned rectangle of size $m \times n$ concentric with $G$. Explicitly, its vertices are the points $(x, y)$ with $x \in \{1/2, m + 1/2\}$ and $y \in \{1/2, n + 1/2\}$.

One construction defined in terms of $G$ will be particularly useful to us later on, and so we introduce special notation for it here. Let $Q$ be the frame of $G$. Denote by $\Phi_{m, n}$ the group of rigid motions in the plane generated by the four reflections with respect to the four sides of $Q$. Given an object $O$ in the plane, such as a point, a straight-line segment, or a polygon, we write $\Phi_{m, n}(O)$ for the set of the images of $O$ under all motions of $\Phi_{m, n}$. For example, if $V$ is the vertex set of $G$, then $\Phi_{m, n}(V)$ is a partitioning of the set of all integer points in the plane.

%% file: sign-03-trans.tex
\section{Transversals} \label{trans}

Let $G$ be a planar graph and let $C_1$ and $C_2$ be two cycles in it with the same signature. We construct a new planar graph $T$ based on $G$, $C_1$, and $C_2$, as follows:

Let $E_1$ be the set of all edges of $G$ which are part of $C_1$ but not $C_2$, define $E_2$ similarly, and let $E = C_1 \Delta C_2 = E_1 \cup E_2$. Fix one point (but not an endpoint) on each edge of $E$. These points are the vertices of $T$.

Let $F$ be a finite face of $G$. Then the same number of edges of $F$ are in $E_1$ and $E_2$; say that there are $k$ of each. Construct $k$ pairwise disjoint curves all of which lie strictly inside of $F$ and each one of which joins two vertices of $T$, one on an edge of $E_1$ and one on an edge of $E_2$. These curves are the edges of $T$.

(It is straightforward to see that we can always pair up the vertices of $T$ on the boundary of $F$ in this way. Indeed, let $L$ be a list of these vertices in the order in which they occur on the boundary of $F$. Find two consecutive vertices in $L$ such that one of them lies on an edge of $E_1$ and the other one lies on an edge of $E_2$, pair them up, and delete them from $L$. Then keep doing this until $L$ becomes empty.)

\begin{figure}[ht] \null \hfill \begin{subfigure}{65pt} \centering \includegraphics{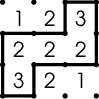} \caption{} \label{trans-a} \end{subfigure} \hspace{30pt} \begin{subfigure}{65pt} \centering \includegraphics{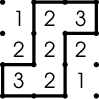} \caption{} \label{trans-b} \end{subfigure} \hspace{30pt} \begin{subfigure}{65pt} \centering \includegraphics{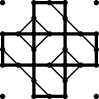} \caption{} \label{trans-c} \end{subfigure} \hfill \null \caption{} \label{trans-fig} \end{figure}

For example, Figure \ref{trans-fig} shows this for two cycles with the same signature in the $4 \times 4$ grid graph.

Clearly, each vertex of $T$ on the boundary of the exterior face of $G$ is of degree $1$ and all other vertices of $T$ are of degree $2$. So $T$ is the disjoint union of some number of paths and cycles. We call each such path or cycle a \emph{transversal}.

Each vertex of $T$ is a spot where an edge of $G$ and a transversal intersect one another. Notice that, by construction, the edges of $G$ that a transversal intersects alternate between $E_1$ and $E_2$.

\begin{lemma} \label{t4} The number of edges of $G$ that a transversal intersects is always divisible by $4$. \end{lemma} 

\begin{myproof} Consider any transversal $t$ which intersects $k_1$ edges of $E_1$ and $k_2$ edges of $E_2$. Since these intersections alternate between $E_1$ and $E_2$, we get that $|k_1 - k_2| \le 1$.

Observe that $t$ partitions the interior of $G$ into two regions. Let $R$ be one of them. Since $C_1$ enters and exits $R$ the same number of times, $k_1$ must be even. Similarly, $k_2$ must be even, too. But then $|k_1 - k_2| \le 1$ implies that $k_1 = k_2$. \end{myproof}

We go on to some applications of transversals.

\begin{theorem} \label{len} Two cycles with the same signature are always of the same length. \end{theorem}

\begin{myproof} By Lemma \ref{t4}, each transversal intersects the same number of edges of $E_1$ and~$E_2$. Summing over all transversals, we see that $|E_1| = |E_2|$. \end{myproof}

\begin{corollary} \label{ext} Two cycles with the same signature, so that they traverse the same number of edges of each finite face, always traverse the same number of edges of the exterior face as well. \end{corollary}

\begin{myproof} This is an equivalent restatement of Theorem \ref{len}, as the sum of the clues of all faces, including the exterior one, counts every edge of the cycle twice, and so equals twice the cycle's length.

(We can also derive this directly from Lemma \ref{t4}. By it, each path transversal connects two vertices of $T$ on two edges of the exterior face which belong one each to $E_1$ and $E_2$. This induces a bijection between the edges of the exterior face in $E_1$ and $E_2$.) \end{myproof}

A different proof of Corollary \ref{ext} which does not rely on transversals was found by Mebane. \cite{B2} For simplicity, here the argument has been modified so that it refers to the equivalent Theorem \ref{len} instead.

\begin{myproof}[Alternative proof of Theorem \ref{len} and Corollary \ref{ext}] Note that $C_1$ partitions the plane into an interior region $I_1$ and an exterior region $O_1$, and define $I_2$ and $O_2$ similarly. Let $R_1 = I_1 \cap O_2$, $R_2 = I_1 \cap I_2$, $R_3 = O_1 \cap I_2$, and $R_4 = O_1 \cap O_2$. Let also $r_{i, j}$ denote the number of edges of $G$ which separate a face in $R_i$ and a face in $R_j$.

From the point of view of $C_1$, the sum of all clues in $R_1$ equals the number of edges of $C_1$ which separate a face in $R_1$ and a face in one of $R_3$ and $R_4$. This is $r_{1, 3} + r_{1, 4}$. Similarly, from the point of view of $C_2$, the sum of all clues in $R_1$ equals $r_{1, 2} + r_{1, 3}$. So $r_{1, 2} = r_{1, 4}$. An analogous argument involving the sum of all clues in $R_3$ shows that $r_{2, 3} = r_{3, 4}$ as well. Hence, $|E_1| = r_{1, 4} + r_{2, 3} = r_{1, 2} + r_{3, 4} = |E_2|$. \end{myproof}

One more application of transversals is the following:

\begin{theorem} \label{d4} The size of the difference between two cycles with the same signature is always divisible by $4$. \end{theorem} 

\begin{myproof} By Lemma \ref{t4}, after we sum over all transversals. \end{myproof}

The divisibility result of Theorem \ref{d4} cannot be improved, in the sense that every positive integer multiple of $4$ does occur, in some planar graph, as the size of the difference between some pair of cycles with the same signature.

Indeed, let $k$ be a positive integer. Construct a planar graph $G$ as follows: The vertices of $G$ are $p_{i, j}$ with $1 \le i \le 4$ and $1 \le j \le k$. Both indices run cyclically, so that, for example, $p_{1, 1}$ is the same vertex as $p_{5, k + 1}$. The edges of $G$ are $p_{i, j}p_{i + 1, j}$ for all $i$ and $j$, $p_{2, j}p_{4, j}$ for all $j$, and $p_{3, j}p_{1, j + 1}$ for all $j$. Thus there are a total of $4k$ vertices and $6k$ edges in $G$. We embed $G$ in the plane so that the boundary of the exterior face is $p_{1, 1}p_{2, 1}p_{3, 1}p_{1, 2}p_{2, 2}p_{3, 2} \ldots p_{1, k}p_{2, k}p_{3, k}$. Then $p_{1, 1}p_{2, 1}p_{4, 1}p_{3, 1}p_{1, 2}p_{2, 2}p_{4, 2}p_{3, 2} \ldots p_{1, k}p_{2, k}p_{4, k}\allowbreak p_{3, k}$ and $p_{1, 1}p_{4, 1}p_{2, 1}p_{3, 1}p_{1, 2}p_{4, 2}p_{2, 2}p_{3, 2} \ldots p_{1, k}p_{4, k}p_{2, k}p_{3, k}$ are two cycles in $G$ with the same signature and a difference of size $4k$.

Similarly to Theorem \ref{len} and Corollary \ref{ext}, Theorem \ref{d4} also admits a proof without the use of transversals.

\begin{myproof}[Alternative proof of Theorem \ref{d4}] We pick up where the alternative proof of Theorem~\ref{len} and Corollary \ref{ext} left off. By considering the sum of all clues in $R_2$, we see that furthermore $r_{1, 2} = r_{2, 3}$. Hence, the four summands on the right-hand side of the identity $|E| = r_{1, 2} + r_{2, 3} + r_{3, 4} + r_{1, 4}$ are pairwise equal. \end{myproof}

Transversals will play a major role also in Section \ref{div}, where we will see that Theorem~\ref{d4} can be strengthened on rectangular grids.

%% file: sign-04-even.tex
\section{Totally Even Sets} \label{even}

Let $G$ be a planar graph. We say that a subset $E$ of the edges of $G$ is \emph{totally even} when every vertex of $G$ is incident with an even number of edges of $E$ and every finite face of $G$ contains an even number of edges of $E$. (By the same counting argument which shows the equivalence of Theorem \ref{len} and Corollary \ref{ext}, we get that $E$ will then contain an even number of edges of the exterior face of $G$, too.)

Clearly, if two cycles in $G$ have the same signature, then their difference is a totally even set. However, the converse is not always true, in the sense that a totally even subset of $G$ might not occur as the difference of two such cycles.

Observe that the indicator vectors of the totally even subsets of $G$ form a vector space over $\mathbb{F}_2$. For convenience, from now on we will say simply that the subsets themselves form such a space, so that sums in the vector setting correspond to symmetric differences in the subset setting.

Given $G$, we construct a new planar graph $H$ as follows: Fix one point (but not an endpoint) on each edge of $G$, as well as one point strictly inside of each finite face of~$G$. The vertices of $H$ are the vertices of $G$ together with all of these points. Thus the vertices of $H$ correspond to the vertices, edges, and finite faces of $G$. Two vertices of $H$ are joined by an edge of $H$ if and only if one of them corresponds to an edge of $G$ and the other one corresponds either to a vertex of $G$ incident with that edge or to a finite face of $G$ containing that edge.

Then $H$ is bipartite. One part consists of the vertices which correspond to the edges of $G$, and the other one consists of the vertices which correspond to the vertices and finite faces of $G$. We colour these two parts in white and black, respectively.

Each totally even subset $E$ of $G$ corresponds to a subset $V$ of the white vertices of $H$ such that every vertex of $H$ has an even number of neighbours in $V$; furthermore, the correspondence is bijective. These subsets of the vertices of $H$ can be viewed as elements of the kernel over $\mathbb{F}_2$ of the biadjacency matrix of $H$, where the rows and columns of the matrix correspond to the black and white vertices of $H$, respectively. For this reason, we call them \emph{kernel} subsets. So, in summary, the totally even subsets of the edges of $G$ are isomorphic to the kernel subsets of the white vertices of $H$.

In the special case when $G$ is the $m \times n$ grid, $H$ becomes the $(2m - 1) \times (2n - 1)$ grid. The white part of $H$ will then consist of the odd vertices of $H$.

The kernel subsets of grids have already been studied in the literature; see \cite{TV} and~\cite{BL}.

Our Lemma \ref{symm}, about the symmetries of the totally even subsets of square grids, is a direct corollary of Proposition 1 in \cite{TV}. A different argument which, in essence, proves Lemma \ref{symm} is sketched in \cite{B1}.

There is also substantial overlap between some of the results of \cite{TV} and \cite{BL} and our Lemmas \ref{basis} and \ref{ter} which establish bases for the spaces of the totally even subsets of grids.

One particularly important fact in this vicinity is that the nullity of the adjacency matrix of the $M \times N$ grid equals $\gcd(M + 1, N + 1) - 1$. Note that we can immediately derive the nullity of the biadjacency matrix as well: In a bipartite graph with parts $V_1$ and $V_2$ and an adjacency matrix of nullity $\mu$, the nullity of the biadjacency matrix equals $(\mu - |V_1| + |V_2|)/2$ when its rows and columns correspond to the vertices of $V_1$ and $V_2$, respectively.

The biadjacency form of the result is part of Theorem 1 in \cite{TV}; there, it is also noted that, for these matrices, the choice of an underlying field $A$ does not matter. The adjacency form is mentioned briefly (without proof and with an implicit $A$) in Remark~5 at the end of the earlier work \cite{DT}.

The Minesweeper theorem of \cite{GL} we cited in the introduction is a direct corollary of the special case when $M + 1$ and $N + 1$ are relatively prime, with $A = \mathbb{Q}$. Similarly, the sufficiency of condition (b) of Theorem \ref{ambs} for uniqueness is a direct corollary of the special case when $\gcd(M + 1, N + 1) = 2$, with $A = \mathbb{F}_2$.

We turn to the totally even subsets of grids. Throughout the rest of this section, let $G$ be the rectangular grid of size $m \times n$.

\begin{lemma} \label{u} Let $E_1$ and $E_2$ be two totally even subsets of a grid $G$. Suppose that $E_1$ and $E_2$ agree on a side of $G$. Then $E_1$ and $E_2$ agree on all of $G$. \end{lemma} 

(As in the introduction, by $E_1$ and $E_2$ ``agreeing'' on a subgraph of $G$, we mean that each edge of the subgraph is in $E_1$ if and only if it is in $E_2$.)

\begin{myproof} Suppose, for concreteness, that $E_1$ and $E_2$ agree on the lower side of $G$. To confirm that in fact they agree everywhere, simply examine all edges of $G$ one by one, going from left to right and then from bottom to top. (Or, equivalently, in colexicographic order of the coordinates of the edges' midpoints.) \end{myproof}

We proceed to focus on the square case; after that, we will reduce the general rectangular case to it. Suppose, temporarily, that $m = n$ and $G$ is a square grid.

To begin with, we describe one family of especially nice totally even subsets of $G$.

For each $1 \le i \le n - 1$, let $D(i)$ be the rectangle with vertices $(1/2, i + 1/2)$, $(i + 1/2, 1/2)$, $(n - i + 1/2, n + 1/2)$, and $(n + 1/2, n - i + 1/2)$. Then $D(i)$ is concentric with $G$, its sides are all of slope $\pm 1$, its side-lengths are $\sqrt{2}i$ and $\sqrt{2}(n - i)$, and it is inscribed in the frame $Q$ of $G$ so that one of its vertices lies on each side of $Q$ and partitions that side into two segments of lengths $i$ and $n - i$. We call each such rectangle $D(i)$ a \emph{diamond}.

Let $B(i)$ be the set of all edges of $G$ contained inside of $D(i)$. Then $B(i)$ is a totally even subset of $G$. We call these the \emph{diamond} subsets of $G$. Notice that, when $D(i)$ is partitioned into $2i \cdot 2(n - i)$ squares of side $1/\sqrt{2}$, exactly one of the two diagonals of each such square is an edge of $G$, and this accounts for all edges of $G$ in $B(i)$.

\begin{figure}[ht] \centering \includegraphics{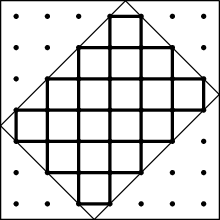} \caption{} \label{diamond} \end{figure}

For example, Figure \ref{diamond} shows $D(3)$ and $B(3)$ on the grid of size $7 \times 7$.

\begin{lemma} \label{basis} On a square grid, the diamond sets form a basis for the space of all totally even sets. \end{lemma}

\begin{myproof} Let $s$ be the lower side of $G$. Each diamond $D(i)$ contains exactly one edge of $s$, namely $(i, 1)$---$(i + 1, 1)$. (So, in particular, no symmetric difference of one or more pairwise distinct diamond sets can be empty.)

Let $E$ be a totally even subset of $G$. For each edge of $E$ in $s$, take the corresponding diamond set. Let $E'$ denote the symmetric difference of all of these diamond sets. Then $E$ and $E'$ agree on $s$. Hence, by Lemma \ref{u}, in fact they agree everywhere. \end{myproof}

By Lemma \ref{basis}, every totally even subset $E$ of $G$ can be expressed uniquely in the form $E = B(a_1) \Delta B(a_2) \Delta \cdots \Delta B(a_k)$ with $a_1 < a_2 < \cdots < a_k$. We call this the \emph{diamond decomposition} of $E$.

\begin{lemma} \label{symm} On a square grid $G$, every totally even set is axially symmetric with respect to both diagonals of $G$ as well as centrally symmetric with the respect to the center of $G$. \end{lemma} 

\begin{myproof} This is clear for the diamond sets. Then for arbitrary totally even sets it follows by Lemma \ref{basis}. \end{myproof}

We are ready to tackle the general rectangular case.

\begin{lemma} \label{ter} Let $m$ and $n$ be positive integers with $d = \gcd(m, n)$, $m = dm'$, and $n = dn'$. Let $G$ be the grid of size $m \times n$ and let $G^\star$ be the grid of size $d \times d$. Given a subset $E^\star$ of the edges of $G^\star$, let $\Psi(E^\star)$ denote the disjoint union of all copies of $E^\star$ in the family $\Phi_{d, d}(E^\star)$ which are contained within $G$. Then the totally even subsets of $G$ are exactly the subsets of the edges of $G$ of the form $\Psi(E^\star)$ for some totally even subset $E^\star$ of $G^\star$. Thus, in particular, the images under $\Psi$ of the diamond subsets of $G^\star$ form a basis for the space of all totally even subsets of $G$. \end{lemma} 

\begin{myproof} It is straightforward to see that, if $E^\star$ is a totally even subset of $G^\star$, then $\Psi(E^\star)$ is a totally even subset of $G$.

For the converse, we proceed by induction on $m' + n'$. The base case $m' = n' = 1$ is clear, and we go on to the induction step.

Suppose, for concreteness, that $m < n$. Let $G_1$ be the grid of size $m \times m$ and let $G_2$ be the translation copy of the grid of size $m \times (n - m)$ by $m$ units upwards. Then the vertex sets of $G_1$ and $G_2$ form a partitioning of the vertex set of $G$.

Let $E$ be a totally even subset of $G$. By the proof of Lemma \ref{basis}, there exists a totally even subset $E_1$ of $G_1$ such that $E$ and $E_1$ agree on the lower side of $G_1$. Then, as in the proof of Lemma \ref{u}, actually $E$ and $E_1$ agree on all of $G_1$. Since each vertex on the upper side of $G_1$ is incident with an even number of edges of $E_1$, it follows that all edges of $G$ between $G_1$ and $G_2$ must be outside of $E$.

Hence, the restriction $E_2$ of $E$ to $G_2$ is a totally even subset of $G_2$. By the induction hypothesis, there exists a totally even subset $E^\star$ of $G^\star$ such that $E_2$ coincides with the disjoint union of the copies of $E^\star$ in the family $\Phi_{d, d}(E^\star)$ which are contained within $G_2$. Then $E$ and $\Psi(E^\star)$ agree on the upper side of $G$. Consequently, by Lemma \ref{u}, in fact they agree on all of $G$ as well. \end{myproof}

We continue with a description of the possible sizes of the totally even subsets of $G$. Once again, we focus on the square case first and then we reduce the general rectangular case to it. Suppose, temporarily, that $m = n$ and $G$ is a square grid.

\begin{lemma} \label{4uv} The size of a totally even subset of the $n \times n$ grid is always of the form $4uv$ with $u + v = n$, for some nonnegative integers $u$ and $v$. Conversely, every number of this form does occur as the size of some such subset. \end{lemma}

\begin{myproof} Sufficiency is straightforward: When both of $u$ and $v$ are positive, the diamond set $B(u)$ is of size $4uv$. We go on to necessity.

Let $E$ be a totally even subset of $G$. The frame $Q$ of $G$ is partitioned by its diagonals into four triangles. Since $E$ is axially symmetric with respect to both of these diagonals by Lemma \ref{symm}, we get that each one of these triangles contains exactly $1/4$ of the edges of~$E$. Hence, it suffices to count the edges of $E$ contained inside of one triangle. We choose the lowermost one, and we denote it by $\Theta$.

Let $E = B(a_1) \Delta B(a_2) \Delta \cdots \Delta B(a_k)$ be the diamond decomposition of $E$, with $a_1 < a_2 < \cdots < a_k$. For convenience, denote $D_i = D(a_i)$.

Let $s$ be the straight-line segment $(1/2, 1/2)$---$(n + 1/2, 1/2)$; this is the lower side of $Q$ as well as base of $\Theta$. The lowermost vertices of $D_1$, $D_2$, $\ldots$, $D_k$ partition $s$ into subsegments. Let these subsegments be $s_1$, $s_2$, $\ldots$, $s_{k + 1}$, in order from left to right. Let also $b_i$ be the length of $s_i$. So $b_1 = a_1$, $b_i = a_i - a_{i - 1}$ for all $2 \le i \le k$, and $b_{k + 1} = n - a_k$. Of course, $b_1 + b_2 + \cdots + b_{k + 1} = n$.

The unit-slope sides of $D_1$, $D_2$, $\ldots$, $D_k$ partition $\Theta$ into forward-slanting strips. Denote the strip containing $s_i$ by $\sigma_i$. Similarly, the sides with slope $-1$ of $D_1$, $D_2$, $\ldots$, $D_k$ partition $\Theta$ into backward-slanting strips. Denote the strip containing $s_i$ by $\tau_i$. So each one of $\sigma_i$ and $\tau_i$ is of width $b_i/\sqrt{2}$.

Let $R_{i, j} = \sigma_i \cap \tau_j$. Then the boundaries of $D_1$, $D_2$, $\ldots$, $D_k$ partition $\Theta$ into the regions $R_{i, j}$ with $i \le j$.

Observe that $R_{i, j}$ is contained inside of $D_h$ if and only if $i \le h < j$. On the other hand, an edge of $G$ is in $E$ if and only if it is contained inside of an odd number of diamonds out of $D_1$, $D_2$, $\ldots$, $D_k$. Hence, if $i \equiv j \pmod 2$, then no edges of $G$ contained inside of $R_{i, j}$ belong to $E$; and, if $i \not\equiv j \pmod 2$, then all edges of $G$ contained inside of $R_{i, j}$ belong to $E$.

In the latter case, $R_{i, j}$ is a rectangle with side-lengths $b_i/\sqrt{2}$ and $b_j/\sqrt{2}$. Thus it contains exactly $b_ib_j$ edges of $G$.

We conclude that the number of edges of $E$ in $\Theta$ equals the sum of $b_ib_j$ over all pairs of indices $i$ and $j$ such that $i < j$ and $i \not\equiv j \pmod 2$. Setting $u = b_1 + b_3 + b_5 + \cdots$ and $v = b_2 + b_4 + b_6 + \cdots$, this sum factors as $uv$. The total number of edges of $E$ then becomes $4uv$, as needed. \end{myproof}

The general rectangular case now follows.

\begin{corollary} \label{4uvmn} Let $m$ and $n$ be positive integers with $d = \gcd(m, n)$, $m = dm'$, and $n = dn'$. Then the size of a totally even subset of the $m \times n$ grid is always of the form $4uvm'n'$ with $u + v = d$, for some nonnegative integers $u$ and $v$. Conversely, every number of this form does occur as the size of some such subset. \end{corollary}

\begin{myproof} By Lemmas \ref{ter} and \ref{4uv}. \end{myproof}

Since kernel sets are of independent interest outside of their connection with signatures, we state the analogue of Lemma \ref{4uv} and Corollary \ref{4uvmn} in this setting explicitly.

\begin{theorem} \label{kss} Let $M$ and $N$ be positive integers with $d = \gcd(M + 1, N + 1)$, $M + 1 = dM'$, and $N + 1 = dN'$. Then the size of a kernel subset of either the even or the odd vertices of the $M \times N$ grid is always of the form $uvM'N'$ with $u + v = d$, for some nonnegative integers $u$ and $v$. Furthermore, when it is a subset of the odd vertices of the grid and $d$ is even, both of $u$ and $v$ must be even as well. Conversely, every number which satisfies these conditions does occur as the size of some such subset. \end{theorem} 

Notice that this is slightly stronger than the original Corollary \ref{4uvmn}. Our isomorphism between totally even sets and kernel sets only involves the kernel subsets of the odd vertices of the grids with odd width and odd height. Still, the proof of the full Theorem~\ref{kss} is analogous to the proofs of Lemma \ref{4uv} and Corollary \ref{4uvmn}; the generalisation does not pose substantial additional difficulties.

%% file: sign-05-const.tex
\section{Constructions} \label{const}

Here we collect some useful constructions.

\begin{lemma} \label{coin} Consider two cycles on a rectangular grid with the same signature. Suppose that they agree on some side of the grid. Then they coincide. \end{lemma} 

\begin{myproof} By Lemma \ref{u}, applied to the empty set and the difference between the two cycles. \end{myproof}

\begin{lemma} \label{side} Let $C_1$ and $C_2$ be two distinct cycles with the same signature on the rectangular grid $G$. Then each side of $G$ contains an edge which belongs to both of $C_1$ and $C_2$. \end{lemma}

\begin{myproof} Consider, for concreteness, the lower side of $G$. By Lemma \ref{coin}, some edge on it is in $E = C_1 \Delta C_2$. Let $k$ be the smallest positive integer such that $e_1 = (k, 1)$---$(k + 1, 1)$ is in $E$. Then $e_2 = (k, 1)$---$(k, 2)$ must be in $E$, too.

\smallskip

\emph{Case 1}. One of $e_1$ and $e_2$ is part of $C_1$ and the other one is part of $C_2$. Then $k \ge 2$ and the edge $(k - 1, 1)$---$(k, 1)$ belongs to both of $C_1$ and $C_2$.

\smallskip

\emph{Case 2}. Both of $e_1$ and $e_2$ are part of the same cycle, without loss of generality $C_1$. Then the other two edges of the cell $(k + 1/2, 3/2)$ must be part of $C_2$. Consequently, $k \le m - 2$ and the edge $(k + 1, 1)$---$(k + 2, 1)$ belongs to both of $C_1$ and $C_2$. \end{myproof}

\begin{lemma} \label{lift} Suppose that there exist two distinct cycles on the $m \times n$ grid with the same signature and a difference of size $s$. Let $a$ and $b$ be any positive integers. Then there exist two distinct cycles on the $am \times bn$ grid with the same signature and a difference of size $abs$. \end{lemma}

\begin{myproof} Let $G$ be the $m \times n$ grid and let $C_1$ and $C_2$ be two cycles on it as specified. Let also $H$ be the $am \times bn$ grid. There are $ab$ copies of $G$ in $\Phi_{m, n}(G)$ which are subgraphs of $H$. We call them \emph{tiles}.

Inside of each tile $\varphi(G)$ with $\varphi \in \Phi_{m, n}$, we place the cycle $\varphi(C_1)$ when $\varphi$ is orientation-preserving and the cycle $\varphi(C_2)$ when $\varphi$ is orientation-reversing. (So the tiles with a copy of $C_1$ and the tiles with a copy of $C_2$ form a checkerboard pattern.) The union of all such cycles is a pseudocycle in $H$ which we denote by $P_1$. Construct $P_2$ similarly, except that ``orientation-preserving'' and ``orientation-reversing'' are swapped in the definition.

Observe that $P_1$ and $P_2$ have the same signature on $H$. Indeed, let $c$ be any cell of~$H$. There are the following cases to consider:

\smallskip

\emph{Case 1}. The cell $c$ lies inside of some tile. Then $P_1$ and $P_2$ contain the same number of edges of $c$ because the signatures of $C_1$ and $C_2$ on $G$ coincide.

\smallskip

\emph{Case 2}. The cell $c$ lies between two tiles so that it is adjacent by side to both of them. Then the two edges of $c$ which lie on the sides of these tiles are images of the same edge $e$ of $G$ under the two corresponding motions of $\Phi_{m, n}$. When $e \in C_1 \Delta C_2$, each one of $P_1$ and $P_2$ contains one edge of $c$. Otherwise, when $e \not\in C_1 \Delta C_2$, clearly $P_1$ and $P_2$ contain the same number of edges of $c$.

\smallskip

\emph{Case 3}. The cell $c$ lies between four tiles so that it is adjacent by corner to all of them. Then neither one of $P_1$ and $P_2$ contains any edges of $c$.

\smallskip

Observe also that $P_1 \Delta P_2$ is of size $abs$, as it is simply the disjoint union of $ab$ copies of $C_1 \Delta C_2$ under the corresponding motions of $\Phi_{m, n}$.

We are only left to modify $P_1$ and $P_2$ so that they become true cycles instead of pseudocycles.

\begin{figure}[ht] \null \hfill \begin{subfigure}{125pt} \centering \includegraphics{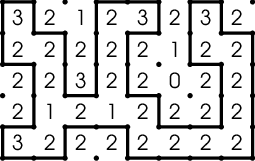} \caption{} \label{lift-a} \end{subfigure} \hfill \begin{subfigure}{125pt} \centering \includegraphics{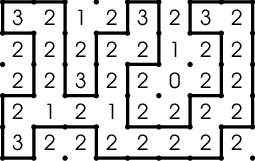} \caption{} \label{lift-b} \end{subfigure} \hfill \null \caption{} \label{lift-fig} \end{figure}

Consider any pair of tiles $G_1$ and $G_2$ which are neighbours by side. Let $t_1$ and $t_2$ be their neighbouring sides. Then $t_1$ and $t_2$ are images of the same side $t$ of $G$ under the two corresponding motions of $\Phi_{m, n}$. By Lemma \ref{side}, there exists some edge of $G$ on $t$ which belongs to both of $C_1$ and $C_2$. Let $e_1$ and $e_2$ be its corresponding edges in $G_1$ and $G_2$, and let $w$ be the cell of $H$ between $G_1$ and $G_2$ which contains $e_1$ and $e_2$.

We call $w$ a \emph{switch} for $G_1$ and $G_2$. To \emph{flip} the switch $w$ is to delete $e_1$ and $e_2$ from $P_1$ and $P_2$ and to replace them with the other pair of opposite edges of $w$. Flipping just one switch in isolation has the effect of splicing together the copies of $C_1$ and $C_2$ inside of $G_1$ and $G_2$, for each one of $P_1$ and $P_2$.

Clearly, by flipping some switches in an appropriate manner, we can splice together all of the cycles of each one of $P_1$ and $P_2$ so that both of $P_1$ and $P_2$ become true cycles. Explicitly, one way to do this would be as follows: Fix any spanning tree of the graph whose vertices are the tiles and whose edges join the pairs of tiles which are neighbours by side. Then, for each edge of that tree, flip one switch between the two tiles joined by that edge.

For example, Figure \ref{lift-fig} shows this construction with $m = n = 3$, $a = 3$, $b = 2$, and $C_1$ and $C_2$ being the two cycles in Figures \ref{amb-a} and \ref{amb-b}, respectively. \end{myproof}

\begin{lemma} \label{cs} Let $n$ be a positive integer. Let also $u$ and $v$ be positive integers with $u$ even and $u + v = n$. Then there exist two cycles on the $n \times n$ grid with the same signature and a difference of size $4uv$. \end{lemma} 

\begin{myproof} Let $G$ be the $n \times n$ grid. Consider the diamond subset $B(u)$ of $G$. We will construct two cycles $C_1$ and $C_2$ on $G$ with the same signature and difference $B(u)$.

Let $p_1$, $p_2$, $\ldots$, $p_u$ be the vertices of $G$ on the south-west side of the diamond $D(u)$, in order as we go south-east, so that $p_i = (i, u - i + 1)$. Similarly, let $q_1$, $q_2$, $\ldots$, $q_u$ be the vertices of $G$ on the north-east side of $D(u)$, in order as we go south-east, so that $q_i = (n - u + i, n - i + 1)$.

For each odd cell $(x, y)$ of $G$ contained inside of $D(u)$, take its lower and right edges when $\lfloor x \rfloor + \lfloor y \rfloor \equiv 1 \pmod 4$, and its left and upper edges otherwise, when $\lfloor x \rfloor + \lfloor y \rfloor \equiv 3 \pmod 4$. Construct also the path $p_1{\sim}(1, 1){\sim}p_u$; the paths $p_i$---$(i, u - i)$---$p_{i + 1}$ for all even $i$ with $2 \le i < u$; and the paths $q_i$---$(n - u + i + 1, n - i + 1)$---$q_{i + 1}$ for all odd $i$.

\begin{figure}[ht] \null \hfill \begin{subfigure}{95pt} \centering \includegraphics{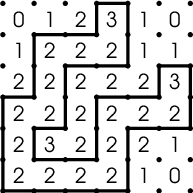} \caption{} \label{cs-a} \end{subfigure} \hfill \begin{subfigure}{95pt} \centering \includegraphics{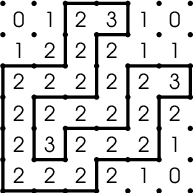} \caption{} \label{cs-b} \end{subfigure} \hfill \null \caption{} \label{cs-fig} \end{figure}

Finally, let $C_1$ be the resulting cycle on $G$ and let $C_2 = C_1 \Delta B(u)$.

For example, Figure \ref{cs-fig} shows this construction with $n = 7$, $u = 4$, and $v = 3$. \end{myproof}

\begin{lemma} \label{c22} Let $m$ and $n$ be positive integers such that $m \ge 4$, $n \ge 4$, both of $m$ and $n$ are even, and at least one of $m$ and $n$ is divisible by $4$. Then there exist two cycles on the $m \times n$ grid with the same signature and a difference of size $mn$. \end{lemma} 

\begin{myproof} When $m = n = 4$, take the cycles $C_1 = (1, 1){\sim}(3, 1)$---$(3, 2)$---$(4, 2){\sim}(4, 4){\sim}\allowbreak(2, 4)$---$(2, 3)$---$(1, 3){\sim}(1, 1)$ and $C_2 = C_1 \Delta B(1) \Delta B(3)$.

\begin{figure}[ht] \null \hfill \begin{subfigure}{140pt} \centering \includegraphics{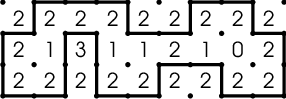} \caption{} \label{c22-a} \end{subfigure} \hfill \begin{subfigure}{140pt} \centering \includegraphics{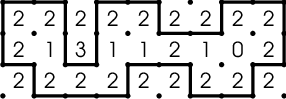} \caption{} \label{c22-b} \end{subfigure} \hfill \null \caption{} \label{c22-fig} \end{figure}

When $m \equiv 2 \pmod 4$ and $n = 4$, let $P$ be the pseudocycle formed as the disjoint union of the cycle $(1, 1){\sim}(3, 1){\sim}(3, 3)$---$(4, 3){\sim}(4, 1){\sim}(6, 1){\sim}(6, 3)$---$(5, 3)$---$(5, 4){\sim}(2, 4)\allowbreak$---$(2, 3)$---$(1, 3){\sim}(1, 1)$ and all translation copies of the previously defined $C_2$ by $k$ units to the right with $k \equiv 2 \pmod 4$ and $6 \le k < m$. For each cell $c = (x, 5/2)$ of the grid such that $x \equiv 5/2 \pmod 4$ and $x \ge 13/2$, delete from $P$ the two vertical edges of $c$ and replace them with the two horizontal edges of $c$. Then take the true cycle thus obtained and its reflection in the horizontal axis of symmetry of the grid.

For example, Figure \ref{c22-fig} shows the preceding construction in the case when $m = 10$.

The remaining cases all follow from these two by Lemma \ref{lift}. \end{myproof}

%% file: sign-06-size.tex
\section{Sizes} \label{size}

Here, for every rectangular grid we determine whether it admits an ambiguous signature or not.

\begin{theorem} \label{ambs} Let $m$ and $n$ be positive integers with $d = \gcd(m, n)$. Then every cycle on the $m \times n$ grid is uniquely determined by its signature if and only if $m$ and $n$ satisfy one or more of the following conditions:

(a) $m \le 2$ or $n \le 2$;

(b) $d = 1$;

(c) $d = 2$ and $m \equiv n \equiv 2 \pmod 4$.

Conversely, there exist two distinct cycles on the $m \times n$ grid with the same signature if and only if $m$ and $n$ satisfy all three of the following conditions:

(a$'$) $m \ge 3$ and $n \ge 3$;

(b$'$) $d \ge 2$;

(c$'$) If $d = 2$, then furthermore either $m \equiv 0 \pmod 4$ or $n \equiv 0 \pmod 4$. \end{theorem} 

Most of Theorem \ref{ambs} is a direct corollary of results that we have established already. The only exception is the sufficiency of condition (c) for uniqueness. (Or, equivalently, the necessity of condition (c$'$) for ambiguity.) In this section, we will obtain it with the help of one more general observation about Slitherlink puzzles. Notice, however, that this part of Theorem \ref{ambs} also follows immediately by Corollary \ref{4uvmn} and Theorem \ref{d8}.

\begin{proposition} \label{22} Let $m$ and $n$ be positive integers with $m \equiv n \equiv 2 \pmod 4$. Consider the Slitherlink puzzle on the $m \times n$ grid where each cell $(x, y)$ such that $x \equiv y \equiv 3/2 \pmod 2$ is clued with a $2$. This Slitherlink puzzle does not admit any solutions. \end{proposition}

\begin{myproof} Let $G$ be that grid and let $P$ be that Slitherlink puzzle. Suppose, for the sake of contradiction, that $C$ is a solution to $P$.

Construct the graph $H$ as follows: The vertices of $H$ are all points $(x, y)$ with $1/2 \le x \le m + 1/2$, $1/2 \le y \le n + 1/2$, and $x \equiv y \equiv 1/2 \pmod 2$. The edges of $H$ are the straight-line segments joining the pairs $p_1 = (x_1, y_1)$ and $p_2 = (x_2, y_2)$ of these points such that, firstly, $|x_1 - x_2| = |y_1 - y_2| = 2$; and, secondly, writing $c$ for the cell of $G$ centered at $(p_1 + p_2)/2$, the two edges of $c$ which belong to $C$ are reflections of one another with respect to the straight line through $p_1$ and $p_2$.

Then the vertices of $H$ at the four points $(x, y)$ with $x \in \{1/2, m + 1/2\}$ and $y \in \{1/2, n + 1/2\}$ are of degree $1$ and all other vertices of $H$ are of even degree.

Colour a vertex $(x, y)$ of $H$ white when $x + y \equiv 1 \pmod 4$ and black when $x + y \equiv 3 \pmod 4$. Then each edge of $H$ joins two vertices of the same colour, and so each connected component of $H$ is monochromatic as well. Since each connected component of $H$ contains an even number of odd-degree vertices, we conclude that a white path in $H$ connects the pair $\{(1/2, 1/2), (m + 1/2, n + 1/2)\}$ and a black path in $H$ connects the pair $\{(m + 1/2, 1/2), (1/2, n + 1/2)\}$. However, two edges of $H$ joining vertices of opposite colours can never intersect, and there is no way for these two paths to get past each other within $H$. We have arrived at a contradiction. \end{myproof}

A different proof of Proposition \ref{22} was found by Mebane. (See the solution to exercise 420 of \cite{K2}.) We reproduce it below.

\begin{myproof}[Alternative proof of Proposition \ref{22}] Define $G$, $P$, and $C$ as before.

Let $S$ be the set of all odd cells of $G$ contained inside of the region enclosed by $C$. Then each clued cell is adjacent by side to two cells of $S$.

Let $T_i$ be the set of all clued cells $(x, y)$ with $x + y \equiv i \pmod 4$, where $i \in \{1, 3\}$. Then each odd cell of $G$ is adjacent by side to one cell of $T_1$ and one cell of $T_3$.

Hence, $|S| = 2|T_1|$ and $|S| = 2|T_3|$. However, the total number of clued cells, $mn/4$, is odd, and so we must have that $|T_1| \neq |T_3|$. \end{myproof}

We go on to the proof of Theorem \ref{ambs}.

\begin{myproof}[Proof of Theorem \ref{ambs}] For uniqueness, the sufficiency of condition (a) is clear; the sufficiency of condition (b) follows by Lemma \ref{ter}; and the sufficiency of condition (c) follows by Lemma \ref{ter} and Proposition \ref{22}.

Conversely, the sufficiency of the conjunction of conditions (a$'$), (b$'$), and (c$'$) for ambiguity follows by Lemmas \ref{lift} and \ref{cs} when $d \ge 3$; and it follows by Lemma \ref{c22} when $d = 2$. \end{myproof}

%% file: sign-07-div.tex
\section{Divisibility} \label{div}

Here we establish one divisibility result, as follows:

\begin{theorem} \label{d8} The size of the difference between two cycles with the same signature on a rectangular grid is always divisible by $8$. \end{theorem} 

By Theorem \ref{d4}, the size of the difference between two solutions to the same signature is always divisible by $4$ already in the setting of general planar graphs. Theorem \ref{d8} shows that a higher divisibility holds on rectangular grids.

Just as with Theorem \ref{d4}, the divisibility result of Theorem \ref{d8} cannot be improved, in the sense that every positive integer multiple of $8$ does occur, on some rectangular grid, as the size of the difference between some pair of cycles with the same signature. This follows by Lemma \ref{cs}.

We will focus on the square case first. Once we have resolved it completely, we will reduce the general rectangular case to it.

The proof relies on a series of lemmas delving successively deeper into the parity properties of transversals on grids.

Let $G$ be the grid of size $n \times n$, let $C_1$ and $C_2$ be two distinct cycles on $G$ with the same signature, let $E = C_1 \Delta C_2$, and let $E = B(a_1) \Delta B(a_2) \Delta \cdots \Delta B(a_k)$ be the diamond decomposition of $E$, with $a_1 < a_2 < \cdots < a_k$. For convenience, denote $D_i = D(a_i)$, as in the proof of Lemma \ref{4uv}.

Construct a transversal graph $T$ for $C_1$ and $C_2$ as in Section \ref{trans}. We can assume without loss of generality that the vertices of $T$ are the midpoints of their respective edges of $G$ and that all edges of $T$ are straight-line segments.

Each vertex of one of $D_1$, $D_2$, $\ldots$, $D_k$ is the apex of an isosceles right-angled triangle whose base is an edge of $E$ on some side of $G$. We attach each such vertex to the transversal containing the midpoint of that edge of $E$. So now each path transversal ``connects'' two diamond vertices attached to it at its endpoints.

Colour the boundary of $D_i$ green when $a_i$ is odd and red when $a_i$ is even.

\begin{lemma} \label{ep} Each path transversal connects two diamond vertices of the same colour. \end{lemma} 

\begin{proof} Consider any path transversal $t$. Since the cells of $G$ that $t$ visits alternate between even and odd, by Lemma \ref{t4} the two which contain the endpoints of $t$ must be of the same parity. But each green diamond vertex is adjacent to an even cell of $G$ and each red diamond vertex is adjacent to an odd cell of $G$. \end{proof}

(A different argument which does not use Lemma \ref{t4} is given at the beginning of the proof of Lemma \ref{cross}.)

Let $Q$ be the frame of $G$. The boundaries of $D_1$, $D_2$, $\ldots$, $D_k$ partition $Q$ into subregions. We call each such subregion a \emph{room}. The outermost rooms are triangles and the rest of them are all rectangles.

Colour all rooms in black and white in a checkerboard manner so that the triangular ones are white. Then a room is white if and only if it is contained inside of an even number of diamonds out of $D_1$, $D_2$, $\ldots$, $D_k$, and it is black if and only if it is contained inside of an odd number of such diamonds. So the edges of $E$ are exactly the edges of $G$ contained inside of black rooms.

Observe that each edge of $T$ is of length either $1/\sqrt{2}$ or $1$.

Suppose first that $e$ is a short edge of $T$. Then $e$ joins the midpoints of two edges of $E$ with a common vertex. So $e$ is contained inside of a single room.

Suppose now that $e$ is a long edge of $T$. Then $e$ joins the midpoints of two edges of $E$ which coincide with two opposite sides of the cell $c$ of $G$ containing $e$. Furthermore, the other two edges of $c$ cannot be in $E$, by the definition of a transversal graph. It follows that the four sides of $c$ are contained inside of four pairwise distinct rooms whose boundaries come together at the center of $c$. So the two halves of $e$ are contained inside of two different rooms adjacent by corner.

When $c$ is an odd cell, the two diamond boundaries which meet at its center are both green, and when it is an even cell they are both red. We refer to these configurations as a \emph{green crossing} and a \emph{red crossing}, respectively. Our analysis shows that a transversal can only travel between rooms via these crossings.

\begin{lemma} \label{cross} Each path transversal which connects two green diamond vertices passes through an even number of red crossings, and similarly with the colours swapped. \end{lemma}

\begin{myproof} Consider any path transversal $t$ and suppose, for concreteness, that $t$ connects two green diamond vertices.

Let us traverse $t$ from one endpoint to the other. In the beginning, we are contained inside of one green diamond and zero red diamonds. Each time when we pass through a green crossing, we flip our membership with respect to two green diamonds, and similarly for the red crossings. So, at all times when we are not on the boundary of a diamond, we will be contained inside of an odd number of green diamonds and an even number of red diamonds.

(The invariance of the parity of the number of diamonds of each colour that we are contained inside of implies the statement of Lemma \ref{ep} in a way which does not rely on Lemma \ref{t4}.)

The boundaries of the red diamonds partition $Q$ into subregions. We call each such subregion which is contained inside of an even number of red diamonds a \emph{space}. (This is similar to the definition of a white room, except that this time around for the subdivision of $Q$ we take only the red diamonds instead of all of $D_1$, $D_2$, $\ldots$, $D_k$.) Colour all spaces in yellow and blue in a checkerboard manner, so that two spaces which touch at a corner are always of opposite colours.

Our initial observation about parities implies that $t$ travels only within spaces. Furthermore, as we traverse $t$, the points where we flip the colour of our space are exactly the red crossings. We are left to show that the two endpoints of $t$ are contained inside of two spaces of the same colour.

Observe that, for each green diamond, all spaces that its boundary intersects are of the same colour. This applies, in particular, to the spaces which contain its corners. Depending on that colour, we call each green diamond either \emph{yellow-cornered} or \emph{blue-cornered}.

Each time when we pass through a green crossing, we flip our membership with respect to two green diamonds of the same type -- either both yellow-cornered or both blue-cornered. This follows because the type of each one of these diamonds is determined by the colour of the space containing the crossing. So the parity of the number of yellow-cornered green diamonds that we are contained inside of, and the analogous parity for blue-cornered green diamonds, are each preserved individually.

Thus the green diamonds which contain the endpoints of $t$ must be of the same type as well -- either both yellow-cornered or both blue-cornered. Consequently, the two endpoints of $t$ are contained inside of two spaces of the same colour, as needed. \end{myproof}

\begin{lemma} \label{par} Suppose that a path transversal connects the two diamond vertices $(x', y')$ and $(x'', y'')$. Then $x' \equiv x'' \pmod 2$ and $y' \equiv y'' \pmod 2$. \end{lemma} 

\begin{myproof} Let $t$ be that path transversal. Then $t$ intersects $4\ell$ edges of $G$, for some positive integer $\ell$, by Lemma \ref{t4}. Let us adjoin two new cells $c_1$ and $c_{4\ell + 1}$ to $G$, centered at $(x', y')$ and $(x'', y'')$, respectively. Let also $c_2$, $c_3$, $\ldots$, $c_{4\ell}$ be the cells of $G$ that $t$ visits on its way from $c_1$ to $c_{4\ell + 1}$, in order, and let $c_i = (x_i, y_i)$.

Since $t$ alternates between even and odd cells, the parity of $c_i$ depends only on the parity of $i$. Suppose, for concreteness, that $c_i$ is of the same parity as $i$ for all $i$. Then both of the diamond vertices attached to $t$ are green.

When $c_i$ contains one or two short edges of $t$, we get that $c_{i - 1}$ and $c_{i + 1}$ are adjacent by corner, and so $x_{i - 1} \not\equiv x_{i + 1} \pmod 2$ and $y_{i - 1} \not\equiv y_{i + 1} \pmod 2$. Otherwise, when $c_i$ contains a long edge of $t$, we get that $c_{i - 1}$ and $c_{i + 1}$ are adjacent to two opposite sides of $c_i$, and so $x_{i - 1} \equiv x_{i + 1} \pmod 2$ and $y_{i - 1} \equiv y_{i + 1} \pmod 2$.

In the latter case, $c_i$ also contains a crossing. When $i$ is even, $c_i$ is even as well, and so that crossing is red. Conversely, when $i$ is odd, the crossing will be green.

By Lemma \ref{cross}, there are an even number of even $i$ such that $c_i$ contains a crossing. So, in the sequence $c_1$, $c_3$, $c_5$, $\ldots$, $c_{4\ell + 1}$, there are an even number of transitions where the parities of both coordinates remain the same and an even number of transitions where the parities of both coordinates are flipped. The desired result now follows. \end{myproof}

We are ready to tackle Theorem \ref{d8}.

\begin{myproof}[Proof of Theorem \ref{d8} in the square case] Define $u$ and $v$ as in Lemma \ref{4uv} and its proof. Suppose, for the sake of contradiction, that both of $u$ and $v$ are odd. Then $n = u + v$ is even and $a_1 + a_2 + \cdots + a_k \equiv a_1 - a_2 + a_3 - a_4 + \cdots \equiv u \equiv v \pmod 2$ is odd. Let $s$ be the number of odd numbers among $a_1$, $a_2$, $\ldots$, $a_k$. So $s$ is odd as well.

Consider the green diamonds' vertices. The lower and upper sides of $Q$ each contain $s$ of them. By Lemma \ref{par}, no transversal can connect a diamond vertex in this set to a diamond vertex outside of this set. Hence, these $2s$ diamond vertices must be paired up by the transversals.

Since $s$ is odd, we get that some transversal must connect a green diamond vertex on the lower side of $Q$ and a green diamond vertex on the upper side of $Q$. Similarly, another transversal must connect a green diamond vertex on the left side of $Q$ and a green diamond vertex on the right side of $Q$. However, there is no way for these two transversals to get past each other within $Q$. We have arrived at a contradiction. \end{myproof}

\begin{myproof}[Proof of Theorem \ref{d8} in the general rectangular case] Let $m$ and $n$ be positive integers with $d = \gcd(m, n)$, $m = dm'$, and $n = dn'$, and let $E$ be the difference of two distinct cycles on the $m \times n$ grid with the same signature.

Suppose first that at least one of $m'$ and $n'$ is even. Then $8$ divides $4m'n'$, and $4m'n'$ divides $|E|$ by Corollary \ref{4uvmn}.

Suppose, otherwise, that both of $m'$ and $n'$ are odd. By Lemma \ref{lift} with $a = n'$ and $b = m'$, we get that there exist two distinct cycles on the $dm'n' \times dm'n'$ grid with the same signature and a difference of size $m'n'|E|$. By the square case of Theorem \ref{d8}, it follows that $8$ divides $m'n'|E|$. Therefore, $8$ divides $|E|$ as well. \end{myproof}

%% file: sign-08-diff.tex
\section{Differences} \label{diff}

Here, for every rectangular grid we determine the greatest possible size of the difference between two cycles on it with the same signature.

We begin with the square case. Define $\delta_n$ by \[\delta_n = n^2 \bmod 8 = \begin{cases} 1 & \text{$n$ odd}\\ 0 & n \equiv 0 \mymod 4\\ 4 & n \equiv 2 \mymod 4 \end{cases}.\]

\begin{theorem} \label{gds} Let $n$ be a positive integer with $n \ge 2$. Then the greatest possible size of the difference between two cycles on the $n \times n$ grid with the same signature is $n^2 - \delta_n$. \end{theorem} 

\begin{myproof} Let $C_1$ and $C_2$ be two such cycles with $E = C_1 \Delta C_2$. Define $u$ and $v$ as in Lemma \ref{4uv}. Then $|E| = 4uv \le (u + v)^2 = n^2$. On the other hand, $|E|$ must be divisible by $8$ by Theorem \ref{d8}. Thus $|E|$ is bounded from above by the greatest multiple of $8$ which does not exceed $n^2$.

When $n \ge 3$, the bound is attained, for example, by the construction of Lemma \ref{cs} with suitable $u$ and $v$. \end{myproof}

We continue with the general rectangular case. Setting $d = \gcd(m, n)$, define $\delta_{m, n}$ by \[\delta_{m, n} = \begin{cases} 1 & \text{$d$ odd}\\ 0 & d \equiv 0 \mymod 4\\ 0 & d \equiv 2 \mymod 4,\\ & m \ge 4, n \ge 4,\\ & \text{and $mn/d^2$ even}\\ 4 & \text{otherwise} \end{cases}.\]

\begin{theorem} \label{gdr} Let $m$ and $n$ be positive integers with $m \ge 2$, $n \ge 2$, $d = \gcd(m, n)$, $m = dm'$, and $n = dn'$. Then the greatest possible size of the difference between two cycles on the $m \times n$ grid with the same signature is $mn - \delta_{m, n}m'n'$. \end{theorem} 

\begin{myproof} That this is an upper bound is established as in the proof of Theorem \ref{gds}, except that instead of Lemma \ref{4uv} we must refer to Corollary \ref{4uvmn}.

That this upper bound is attained, in the cases where it is nonzero, follows by Lemma~\ref{lift} together with the construction in the proof of Theorem \ref{gds} when $\delta_d = \delta_{m, n}$; and it follows by Lemma \ref{c22} otherwise, when $d \equiv 2 \pmod 4$, $\delta_d = 4$, and $\delta_{m, n} = 0$. \end{myproof}

%% file: sign-09-mult.tex
\section{Multiplicities} \label{mult}

Here we study the multiplicities of signatures.

We begin with general planar graphs. In this setting, every positive integer occurs as a multiplicity.

\begin{theorem} \label{mpg} Let $s$ be a positive integer. Then there exists a signature in a planar graph with multiplicity $s$. \end{theorem} 

\begin{myproof} The case of $s = 1$ is clear, so suppose, from now on, that $s \ge 2$.

Construct a planar graph $G$ as follows: The vertices of $G$ are $p_{i, j}$ with $1 \le i \le 8$ and $1 \le j \le s$. Both indices run cyclically, so that, for example, $p_{1, 1}$ is the same vertex as $p_{9, s + 1}$. The edges of $G$ are $p_{i, j}p_{i + 1, j}$ for all $i$ and $j$, $p_{i, j}p_{i + 2, j}$ for all even $i$ and all $j$, and $p_{3, j}p_{1, j + 1}$ and $p_{5, j}p_{7, j + 1}$ for all $j$. Thus there are a total of $8s$ vertices and $14s$ edges in $G$. We embed $G$ in the plane so that the boundary of the exterior face is $p_{7, 1}p_{6, 1}p_{5, 1}p_{7, 2}p_{6, 2}p_{5, 2} \ldots p_{7, s}p_{6, s}p_{5, s}$.

We clue each triangular face of $G$ of the form $p_{i - 1, j}p_{i, j}p_{i + 1, j}$ with a $1$ when $i \equiv 1 \pmod 4$ and with a $2$ when $i \equiv 3 \pmod 4$; each quadrilateral face of the form $p_{2, j}p_{4, j}p_{6, j}p_{8, j}$ with a $2$; each hexagonal face of the form $p_{3, j}p_{4, j}p_{5, j}p_{7, j + 1}p_{8, j + 1}p_{1, j + 1}$ with a $4$; and the remaining finite face $p_{1, 1}p_{2, 1}p_{3, 1}p_{1, 2}p_{2, 2}p_{3, 2} \ldots p_{1, s}p_{2, s}p_{3, s}$ with a $2s$.

\begin{figure}[ht] \centering \includegraphics{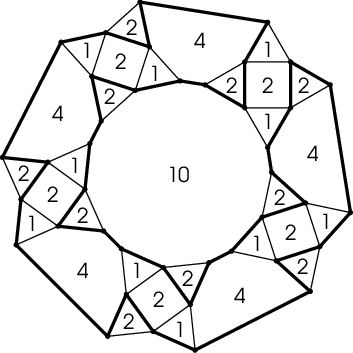} \caption{} \label{wreath} \end{figure}

For example, Figure \ref{wreath} shows this construction in the case when $s = 5$.

Consider any weak solution $C$ to that signature. Then $p_{2, j}p_{8, j}$ is not an edge of $C$ because otherwise the clues of the two triangular faces and the hexagonal face containing $p_{8, j}$ cannot be satisfied simultaneously. Similarly, $p_{4, j}p_{6, j}$ is not an edge of $C$ because of the two triangular faces and the hexagonal face containing $p_{4, j}$. Some straightforward casework now shows that there are only two ways to resolve the edges incident with $p_{1, j}$, $p_{2, j}$, $\ldots$, $p_{8, j}$: Either $p_{3, j - 1}p_{1, j}p_{8, j}p_{6, j}p_{7, j}p_{5, j - 1}$ and $p_{1, j + 1}p_{3, j}p_{2, j}p_{4, j}p_{5, j}p_{7, j + 1}$ are subpaths of $C$, when we say that $j$ is of type I; or $p_{3, j - 1}p_{1, j}p_{2, j}p_{4, j}p_{3, j}p_{1, j + 1}$ and $p_{5, j - 1}p_{7, j}p_{8, j}p_{6, j}p_{5, j}p_{7, j + 1}$ are subpaths of $C$, when we say that $j$ is of type II.

There are the following cases to consider:

\smallskip

\emph{Case 1}. There are no $j$ of type I. Then $C$ is the disjoint union of two cycles of length $4s$ each.

\smallskip

\emph{Case 2}. There are two or more $j$ of type I. Suppose that $j'$ is of type I, all of $j' + 1$, $j' + 2$, $\ldots$, $j'' - 1$ are of type II, and $j''$ is of type I. Then $C$ contains a cycle with vertices $p_{i, j'}$ for $i \in \{2, 3, 4, 5\}$, $p_{i, j''}$ for $i \in \{1, 6, 7, 8\}$, and $p_{i, j}$ for all $i$ and $j$ with $j' < j < j''$. Thus $C$ cannot be a true cycle.

\smallskip

\emph{Case 3}. Exactly one $j$ is of type I. Then $C$ is a true cycle.

\smallskip

Therefore, exactly $s$ cycles in $G$ satisfy our signature, as needed. \end{myproof}

We go on to rectangular grids.

It will be instructive to consider pseudocycles first. Once again, every positive integer occurs as a multiplicity.

\begin{theorem} \label{mp} Let $s$ be a positive integer. Then there exist pseudosignatures on arbitrarily large square grids which are satisfied by exactly $s$ pseudocycles. \end{theorem} 

\begin{myproof} Since the case of $s = 1$ is clear, suppose, from now on, that $s \ge 2$.

Let $n$ be an even positive integer with $n \ge 2s$ and let $G$ be the $n \times n$ grid. Fix two even positive integers $k$ and $\ell$ with $2 \le k \le \ell \le n - 2$ and $\ell - k = 2s - 4$, and denote the convex hull of the diamonds $D(k)$ and $D(\ell)$ by $\Theta$. Let $p_1$, $p_2$, $\ldots$, $p_k$ be the vertices of $G$ on the south-west side of $\Theta$, in order as we go south-east, so that $p_i = (i, k - i + 1)$. Similarly, let $q_1$, $q_2$, $\ldots$, $q_{n - \ell}$ be the vertices of $G$ on the south-east side of $\Theta$, in order as we go north-east, so that $q_i = (\ell + i, i)$.

Take all horizontal edges of $G$ contained inside of $\Theta$. To them, add all edges of $G$ of the form $(1, i)$---$(1, i + 1)$ with $i$ odd and $k < i < \ell$, as well as their images under central symmetry with respect to the center of $G$. Construct also the paths $p_i$---$(i, k - i)$---$p_{i + 1}$ for all odd $i$; the paths $q_i$---$(\ell + i + 1, i)$---$q_{i + 1}$ for all odd $i$; and the images of all of these paths under central symmetry with respect to the center of $G$.

Let $C_1$ be the resulting subgraph of $G$. Furthermore, for all $i$ with $2 \le i \le s$, define $C_i = C_1 \Delta B(k + 2i - 4)$. It is straightforward to see that each $C_i$ is a pseudocycle in $G$ and that all of these pseudocycles have the same pseudosignature $S$.

\begin{figure}[ht] \null \hfill \begin{subfigure}{110pt} \centering \includegraphics{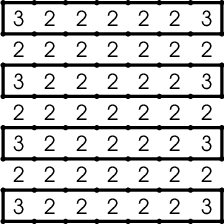} \caption{} \label{pseudo-a} \end{subfigure} \hspace{40pt} \begin{subfigure}{110pt} \centering \includegraphics{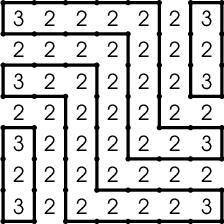} \caption{} \label{pseudo-b} \end{subfigure} \hfill \null\\ \vspace{\baselineskip}\\ \null \hfill \begin{subfigure}{110pt} \centering \includegraphics{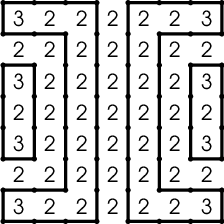} \caption{} \label{pseudo-c} \end{subfigure} \hspace{40pt} \begin{subfigure}{110pt} \centering \includegraphics{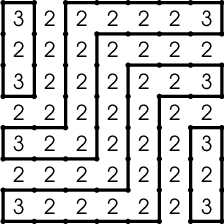} \caption{} \label{pseudo-d} \end{subfigure} \hfill \null \caption{} \label{pseudo} \end{figure}

For example, Figure \ref{pseudo} shows this construction with $s = 4$, $n = 8$, $k = 2$, and $\ell = 6$.

We are left to check that $S$ does not admit any other weak solutions. Let $C$ be a pseudocycle in $G$ with pseudosignature $S$ and let $B(a_1) \Delta B(a_2) \Delta \cdots \Delta B(a_t)$ be the diamond decomposition of $C \Delta C_1$. The clues of $S$ with value $0$ in the lowermost row of $G$ imply that $k - 1 \le a_i \le \ell + 1$ for all $i$.

Suppose, for the sake of contradiction, that there exist two distinct values $u$ and $v$ with $u < v$ among $a_1$, $a_2$, $\ldots$, $a_t$. Let $w$ be the intersection point of the south-east side of $D(u)$ and the south-west side of $D(v)$. There are the following cases to consider:

\smallskip

\emph{Case 1}. The point $w$ is the center of a cell. Then this cell is clued with a $2$ in $S$ but it contains either $0$ or $4$ edges of $C$.

\smallskip

\emph{Case 2}. The point $w$ is a vertex of $G$. Since $w$ is incident with two horizontal edges of $C_1$, either all edges of $G$ incident with $w$ belong to $C$, or none of them do. The former subcase cannot occur because $w$ cannot be of degree $4$ in $C$, and so $C$ does not visit $w$.

Since Theorem \ref{len} applies to pseudocycles as well, $C$ must visit the same number of vertices of $G$ as $C_1$. On the other hand, because of the clues of $S$ with values $0$ and $1$, we get that the vertices of $C$ must be a subset of the vertices of $C_1$. Hence, $C$ cannot miss a vertex visited by $C_1$.

\smallskip

In both cases, we arrive at a contradiction. Thus either $C = C_1$ or $C \Delta C_1 = B(h)$ for some $h$ with $k - 1 \le h \le \ell + 1$. Supposing the latter, if $h$ is odd, then the cell $(3/2, h + 1/2)$ of $G$ will contain only one edge of $C$, whereas it is clued with a $3$ in $S$. Therefore, $h$ must be even and $C$ must coincide with one of $C_2$, $C_3$, $\ldots$, $C_s$. \end{myproof}

We continue with true cycles. In this setting, the question of which positive integers occur as multiplicities becomes much more difficult.

Every grid with at least two vertices along each side admits a signature with multiplicity $1$. For example, the boundary of each face -- including the exterior one -- is a cycle whose signature determines it uniquely.

For multiplicity $2$, by setting $u = 2$ in the construction of Lemma \ref{cs} we get one infinite family of Slitherlink signatures where it is particularly straightforward to verify that no additional solutions are possible. (In fact, all signatures constructed in the proof of Lemma \ref{cs} are of multiplicity exactly $2$. This can be shown as in the proofs of Propositions \ref{m4} and \ref{m3}, but with Lemma \ref{coin} applied to either the right or the upper side of the grid.)

\begin{figure}[ht] \null \hfill \begin{subfigure}{95pt} \centering \includegraphics{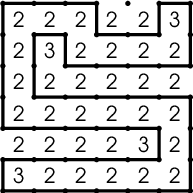} \caption{} \label{four-a} \end{subfigure} \hspace{40pt} \begin{subfigure}{95pt} \centering \includegraphics{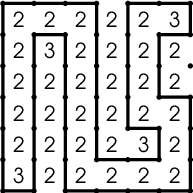} \caption{} \label{four-b} \end{subfigure} \hfill \null\\ \vspace{\baselineskip}\\ \null \hfill \begin{subfigure}{95pt} \centering \includegraphics{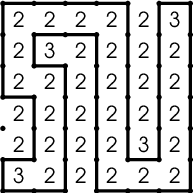} \caption{} \label{four-c} \end{subfigure} \hspace{40pt} \begin{subfigure}{95pt} \centering \includegraphics{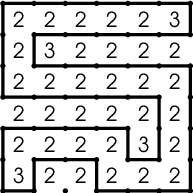} \caption{} \label{four-d} \end{subfigure} \hfill \null \caption{} \label{four} \end{figure}

The next easiest case is multiplicity $4$. An exhaustive computer search demonstrates that the smallest square grid which admits this multiplicity is the one of size $7 \times 7$; furthermore, modulo the symmetries of the grid, there exists a unique such signature on it, shown in Figure \ref{four}. It coincides with the case $k = 1$ of the construction described below.

\begin{proposition} \label{m4} There exist Slitherlink signatures of multiplicity $4$ on arbitrarily large square grids. \end{proposition} 

\begin{myproof} Let $k$ be a positive integer, let $n = 4k + 3$, and let $G$ be the $n \times n$ grid.

Denote by $\Theta$ the convex hull of the diamonds $D(2)$ and $D(3)$. Let $p_1 = (1, 2)$ and $p_2 = (2, 1)$ be the two vertices of $G$ on the south-west side of $\Theta$ and let $q_1$, $q_2$, $\ldots$, $q_{n - 3}$ be the vertices of $G$ on the south-east side of $\Theta$, in order as we go north-east, so that $q_i = (i + 3, i)$.

Take all horizontal edges of $G$ contained inside of $D(2) \cap D(3)$. To them, add all horizontal edges of $G$ contained inside of $\Theta$ which are incident with a vertex of $G$ on one of the south-west and south-east sides of $\Theta$, as well as all vertical edges of $G$ contained inside of $\Theta$ which are incident with a vertex of $G$ on one of the north-west and north-east sides of $\Theta$. Construct also the path $p_1$---$(1, 1)$---$p_2$; the paths $q_i{\sim}(i + 6, i){\sim}q_{i + 3}$ and $q_{i + 1}$---$(i + 5, i + 1)$---$q_{i + 2}$ for all $i$ with $i \equiv 1 \pmod 4$; and the images of all of these paths under central symmetry with respect to the center of $G$.

Let $C_1$ be the resulting subgraph of $G$. Define also $C_2 = C_1 \Delta B(2)$, $C_3 = C_1 \Delta B(3)$, and $C_4 = C_1 \Delta B(2) \Delta B(3)$. It is straightforward to see that $C_1$, $C_2$, $C_3$, and $C_4$ are four cycles on $G$ with the same signature $S$.

We are left to show that $S$ does not admit any other solutions. Let $C$ be any solution to it. The clues of $S$ with value $0$ in the lowermost row of $G$ imply that the edges of the subgraph $(7, 1){\sim}(n, 1)$ of $G$ do not belong to $C$. Once we know this, some casework using the clues of $S$ in the cells $(3/2, 3/2)$, $(11/2, 3/2)$, $(11/2, 5/2)$, and $(13/2, 3/2)$ confirms that $C$ must agree with one of $C_1$, $C_2$, $C_3$, and $C_4$ on the subgraph $(1, 1){\sim}(7, 1)$ of $G$. By Lemma \ref{coin}, we conclude that $C = C_i$ for some $i$. \end{myproof}

\begin{figure}[ht] \null \hfill \begin{subfigure}{110pt} \centering \includegraphics{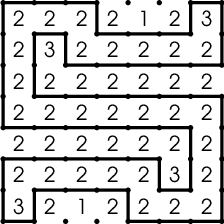} \caption{} \label{three-i-a} \end{subfigure} \hfill \begin{subfigure}{110pt} \centering \includegraphics{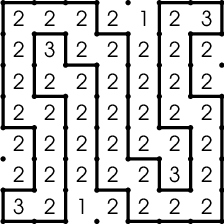} \caption{} \label{three-i-b} \end{subfigure} \hfill \begin{subfigure}{110pt} \centering \includegraphics{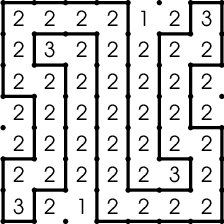} \caption{} \label{three-i-c} \end{subfigure} \hfill \null \caption{} \label{three-i} \end{figure}

\begin{figure}[ht] \null \hfill \begin{subfigure}{110pt} \centering \includegraphics{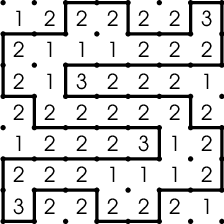} \caption{} \label{three-ii-a} \end{subfigure} \hfill \begin{subfigure}{110pt} \centering \includegraphics{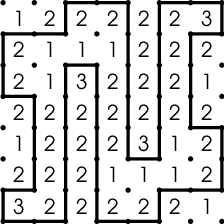} \caption{} \label{three-ii-b} \end{subfigure} \hfill \begin{subfigure}{110pt} \centering \includegraphics{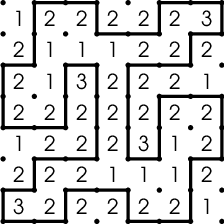} \caption{} \label{three-ii-c} \end{subfigure} \hfill \null \caption{} \label{three-ii} \end{figure}

We turn to multiplicity $3$. An exhaustive computer search demonstrates that the smallest square grid which admits this multiplicity is the one of size $8 \times 8$; furthermore, modulo the symmetries of the grid, there exist exactly two such signatures on it, shown in Figures \ref{three-i} and \ref{three-ii}. The first one of them coincides with the case $k = 1$ of the construction described below.

\begin{proposition} \label{m3} There exist Slitherlink signatures of multiplicity $3$ on arbitrarily large square grids. \end{proposition} 

\begin{myproof} Let $k$ be a positive integer, let $n = 4k + 4$, and let $G$ be the $n \times n$ grid.

Denote by $\Theta$ the convex hull of the diamonds $D(2)$ and $D(4)$. Let $p_1 = (1, 2)$ and $p_2 = (2, 1)$ be the two vertices of $G$ on the south-west side of $\Theta$ and let $q_1$, $q_2$, $\ldots$, $q_{n - 4}$ be the vertices of $G$ on the south-east side of $\Theta$, in order as we go north-east, so that $q_i = (i + 4, i)$.

Take all horizontal edges of $G$ contained inside of $D(2) \cap D(4)$. To them, add the left and upper edges of all cells $(x, y)$ of $G$ with $x + y = 4$; the right and upper edges of all cells $(x, y)$ of $G$ with $x - y = 3$; and the images of all of these edges under central symmetry with respect to the center of $G$. Construct also the path $p_1$---$(1, 1)$---$p_2$; the paths $q_i{\sim}(i + 7, i){\sim}q_{i + 3}$ and $q_{i + 1}$---$(i + 6, i + 1)$---$q_{i + 2}$ for all $i$ with $i \equiv 1 \pmod 4$; and the images of all of these paths under central symmetry with respect to the center of $G$.

Let $C_1$ be the resulting subgraph of $G$. Define also $C_2 = C_1 \Delta B(2)$ and $C_3 = C_1 \Delta B(4)$. It is straightforward to see that $C_1$, $C_2$, and $C_3$ are three cycles on $G$ with the same signature $S$.

The argument that $S$ does not admit any other solutions proceeds along the same lines as in the proof of Proposition \ref{m4}. What changes is that we subdivide the lower side of $G$ into the subgraphs $(1, 1){\sim}(8, 1)$ and $(8, 1){\sim}(n, 1)$, and the casework relies on the clues of $S$ in the cells $(3/2, 3/2)$, $(3/2, 7/2)$, $(5/2, 5/2)$, $(7/2, 3/2)$, $(13/2, 3/2)$, $(13/2, 5/2)$, and $(15/2, 3/2)$. \end{myproof}

%% file: sign-10-further.tex
\section{Further Work} \label{further}

Perhaps the central open question in the area of Slitherlink signatures is that of which positive integers occur as the multiplicity of some Slitherlink signature. Partial advances would be of interest, too, such as a complete description of the grids which admit a signature with multiplicity at least $3$; a proof or disproof that there exists a Slitherlink signature with multiplicity at least $5$; and a proof or disproof that the multiplicities of Slitherlink signatures can be arbitrarily large.

There are many other settings where signatures could be studied, besides rectangular grids. For example, we could consider grids where the cells are unit equilateral triangles or unit regular hexagons rather than unit squares. Or, naturally, instead of various kinds of grids we could explore different classes of planar graphs altogether.

%% file: sign-11-ack.tex
\section*{Acknowledgements} \label{ack}

The author is thankful to Professor Donald Knuth, for inciting him to study signatures in depth; and to Palmer Mebane, for contributing the alternative proofs of Theorem~\ref{len}, Corollary \ref{ext}, and Proposition \ref{22}.